\documentclass[11pt]{amsart}
\usepackage{amstext}
\usepackage{amsfonts}
\usepackage{amssymb}
\usepackage{amsbsy}
\usepackage{amsmath}
\usepackage{latexsym}
\usepackage{hhline}

\newcounter{num}[section] %

\newenvironment{theo}
{\refstepcounter{num}%
\bigskip\noindent{\bf Theorem~\arabic{section}.\arabic{num}. }\it}

\newenvironment{lemma}
{\refstepcounter{num}%
\bigskip\noindent{\bf Lemma~\arabic{section}.\arabic{num}. }\it}

\newenvironment{example}
{\refstepcounter{num}%
\bigskip\noindent{\bf Example~\arabic{section}.\arabic{num}.}}

\newcommand{\Ref}[1]{(\ref{#1})}

\newcounter{thepic}

\newenvironment{eq}{\begin{equation}}{\end{equation}}

\newcommand{\si}{\sigma}
\newcommand{\al}{\alpha}
\newcommand{\be}{\beta}
\newcommand{\ga}{\gamma}

\newcommand{\de}{\delta}

\newcommand{\LA}{\langle}
\newcommand{\RA}{\rangle}

\newcommand{\un}[1]{{\underline{#1}} }

\newcommand{\tr}{\mathop{\rm tr}}

\newcommand{\Char}{\mathop{\rm char}}

\newcommand{\X}{\LA X\RA}

\newcommand{\gl}{{\rm gl}}
\renewcommand{\sl}{{\rm sl}}
\newcommand{\ad}{{\rm ad}}

\newcommand{\FF}{{\mathbb{F}}}   
\newcommand{\NN}{{\mathbb{N}}}
\newcommand{\QQ}{{\mathbb{Q}}}

\newcommand{\algA}{\mathcal{A}}    
\newcommand{\algL}{\mathcal{L}}    
\newcommand{\F}{\mathcal{F}}    

\newcommand{\mylabel}[1]{}
\newcommand{\Fspan}{\FF\text{\rm-span}}

\begin{document}

\title[Identities for $\gl_2$]{Identities for the Lie algebra $\gl_2$ over an infinite field of characteristic two}%
%
%
\author{Artem Lopatin }
\thanks{This research was supported by RFBR 16-31-60111 (mol\_a\_dk).}
\address{Artem Lopatin\\
Sobolev Institute of Mathematics, Omsk Branch, SB RAS, Omsk, Russia}
\email{artem\_lopatin@yahoo.com}

\begin{abstract} 
In 1970 Vaughan-Lee established that over an infinite field of characteristic two the ideal $T[\gl_2]$ of all polynomial identities for the Lie algebra $\gl_2$ is not finitely generated as a T-ideal. But a generating set for this ideal of polynomial identities was not found. We establish 
some generating set for the T-ideal $T[\gl_2]$. 



\end{abstract}

\maketitle
%

\section{Introduction}\label{section_intro}

We consider all vector spaces, algebras and modules over an infinite field $\FF$ of arbitrary characteristic $p=\Char{\FF}$ unless otherwise stated. Denote by $\gl_n$ be the Lie algebra of all $n\times n$ matrices over $\FF$. We write $x_1\cdots x_n$ for the left-normalized product $(\cdots((x_1x_2)x_3)\cdots x_n)$ and define the adjoint $\ad x$ by $y(\ad x)=yx$.

We start with basic definitions. Denote by $\algL(X)$ the free Lie algebra freely generated by $x_1,x_2,\ldots$ Given a Lie algebra $\algL$, an element $f$ of the free Lie algebra $\algL(X)$ is called a {\it polynomial identity} for  $\algL$ if  $f(a_1,\ldots,a_k)=0$ for all $a_1,\ldots,a_k$ from $\algL$, where $f(a_1,\ldots,a_k)$ stands for the result of substitution $x_i\to a_i$ in $f$. In other words, a polynomial identity is a universal formula that holds on the algebra. For short, polynomial identities are called identities. The ideal of all identities for $\algL$ is denoted by $T[\algL]$. If a Lie algerba satisfies a non-trivial identity, then it is called a Lie PI-algebra. An ideal $I\lhd\algL(X)$ is called a {\it T-ideal} if $I$ is stable with respect to all substitutions $x_1\to f_1, x_2\to f_2,\ldots$ for $f_1,f_2,\ldots\in\algL(X)$. Note that  $T[\algL]$ is a T-ideal and any T-ideal $I$ is equal to $T[\algL]$ for some $\algL$. For example, we can take $\algA=\FF\X / I$. A T-ideal $I$ is called {\it finitely based} if it is finitely generated as T-ideal. 

In 1590 Specht posted the following problem for associative algebras: is the T-ideal of identities finitely based for every algebra? (Definfitions can be naturally extended from the case of Lie algebras). In 1986 Kemer established a positive solution for this problem in case $p=0$. On the other hand, in case $p>0$ there exists non-finitely based ideals of identities (see Belov, Grishin, Shchigolev, 1999). 

As about the case of Lie algebras, In 1970 Vaughan--Lee~\cite{VaughanLee_1970} established that over an arbitrary infinite field of the characteristic two there exists a finite dimensional algebra such that its T-ideal of identities is not finitely based. Namely, the T-ideal $T[\gl_2]$ is not finitely based, but a et generating T-ideal $T[\gl_2]$ was not found. In 1974 Drensky~\cite{Drensky_1974} extended this result as follows. For an arbitrary infinite field of positive characteristic he constructed an example of finitely dimensional Lie algebra which T-ideal is not finitely based. In 1992 Il'tyakov~\cite{Iltyakov_1992} established that T-ideal of identities for any finite dimensional Lie algebra over a field of chacteristic zero is finitely based. Note that the following question is still open: does any Lie algebra over a field of characteristic zero have a finitely based T-ideal of identities?

It is an interesting problem to explicitly describe a generating set for the T-ideal of identities of the  particular Lie algebra. The identities for $\gl_2$ and $\sl_2$ were found by several mathematicians over any infinite field with the exception of identities for $\gl_2$ in the characteristic two case. Over a field of characteristic zero Razmyslov~\cite{Razmyslov_1973} found a finite set generating the T-ideal of identities for $\sl_2$. A minimal generating set for the T-ideal $T[\sl_2]$ was established by Filippov~\cite{Filippov_1981}. Working over an infinite field of characteristic different from two Vasilovskii~\cite{Vasilovskii_1989} showed that any identity for $\sl_2$ follows from the single identity
$$yz(tx)x+yx(zx)t=0.$$
Over a field of characteristic different from two identities for $\gl_2$ coincide with identites for $\sl_2$, since $A-\frac{1}{2}\tr(A)E$ satisfies every identity for $\sl_2$, where $A\in\gl_2$. It is not difficult to see that over an arbitrary (finite or infinite) field of characteristic two the ideal of identities for $\sl_2$ is generated by a single identity $xyz$. In 2009 Krasilnikov~\cite{Krasilnikov_2009} showed that the T-ideal of the certain subalgebra of $\gl_3(\QQ)$ is finitely based. 

We obtained an explicit description of identities (the proof is given at the end of the paper). 

\begin{theo}\label{theo-main}\mylabel{theo-main} In the case of an infinite field $\FF$ of the characteristic two the T-ideal of identities for $\gl_2$ is generated by the following identities:
\begin{enumerate}    
\item[(a)] $(x_1x_2)(x_3x_4)x_5$;

\item[(b)] $(x_1x_2)(x_1x_2\cdots x_k)$, where $k>2$;

\item[(c)] $(x_1x_2)(x_3x_4)+(x_1,x_3)(x_2x_4)+(x_1x_4)(x_2x_3)$;

\item[(d)] $(x_1x_2)(x_3x_4\cdots x_m)+(x_1x_3)(x_2x_4\ldots x_m) +(x_1x_4)(x_2x_3\cdots x_m)$, where $m>4$. 
\end{enumerate}
\end{theo}

\section{The known results on identities of $\gl_2$}
In the rest of this paper we assume that $\FF$ is an infinite field of the characteristic two. Denote $\NN_0=\{0,1,2,\ldots\}$. Given two $2\times 2$ matrices $A$ and $B$, we write $AB$ for the Lie product $AB=[A,B]=A\cdot B - B\cdot A$, where $A\cdot B$ is the usual (associative) matrix multiplication. 

Introduce some notations:
$$a=E_{22},\; b=E_{12},\; c=E_{21},$$
where $E_{i,j}$ is $2\times2$ matrix with the only non-zero entry in position $(i,j)$. Then $a$, $b$, $c$, $bc=E_{11}+E_{22}$ is a base for $\gl_2$. Note that 
$$ba=b,\; ca=c,\; (bc)x=0 \text{ for each }x\in \gl_2.$$
Considering identities of $\gl_2$, for short sometimes we write $i$ instead of letter $x_i$. Moreover, we write $1 2 \cdots k$ for $(\cdots((12)3)\cdots )k$. As an example, the Jacobi identity is $123+231+312=0$. 

In this section we consider some facts obtained in~\cite{VaughanLee_1970}.

\begin{lemma}\label{L1e2} \mylabel{L1e2} 
The following identities hold in $\gl_2$:
\begin{enumerate}
\item[(1)] $(12)(34)5 = 0$; 

\item[(2)] $(1\cdots n)(12) = 0$, where $n\geq2$.
\end{enumerate}
\end{lemma}

Denote by $\F$ the quotient of the free Lie algebra with free generators $y_1,y_2,\ldots$ by the ideal generated by the identity $(12)(34)5=0$. Since $(12)(34)5 = 0$ is an identity of $\gl_2$, we can consider the T-ideal of identities of $\gl_2$ as a T-ideal in $\F$; in this case we write down $T[\gl_2]\lhd \F$. The algebra $\F$ has $\NN_0$-grading by the degrees and $\NN_0^n$-grading by multidegrees.
For short, we denote the multidegree $(1,\ldots,1)$ ($n$ times) by $1^n$. An element is called multilinear if it is $\NN_{0}^n$-homogeneous and its degree with respect to each letter is either 0 or 1. Note that an element of multidegree $1^n$ is always multilinear but the inverse statement does not hold. 

\begin{lemma}\label{LFid} \mylabel{LFid} 
The following identities hold in $\F$:
\begin{enumerate}
\item[(1)] $(125)(34) = (12)(345)$; 

\item[(2)] $(xy\, g_1 \cdots g_n)(uv) = (xy\, g_{\si(1)}\cdots g_{\si(r)}) (uv\, g_{\si(r+1)}\cdots g_{\si(n)})$, where  $n\geq1$, $0\leq r\leq n$ and $\si\in S_n$. 
\end{enumerate}
\end{lemma}

\begin{lemma}\label{LF} \mylabel{LF} 
The algebra $\F^2$ satisfies the following properties:
\begin{enumerate}
\item[(1)] $\F^2=\Fspan\{y_{i_1}\cdots y_{i_k}\,|\, k\geq 2,\; i_1>i_2\leq i_3\}$; 

\item[(2)] $\F^2=\Fspan\{y_{i_1}\cdots y_{i_k}\,|\, k\geq 2,\; i_1>i_2\leq i_3\leq \cdots\leq i_k\}$ modulo ($\F^2)^2$; 

\item[(3)] $(\F^2)^2$ is $\FF$-span of $(y_{i_1}\cdots y_{i_k})(y_{i_{k+1}}y_{i_{k+2}})$, where $k\geq 2$ and  
\begin{enumerate}
\item[$\bullet$] $i_1>i_2\leq i_3\leq \cdots\leq i_k$,

\item[$\bullet$] $i_{k+1}>i_{k+2}\leq i_3$, 

\item[$\bullet$] $i_2\geq i_{k+2}$,

\item[$\bullet$] $i_1\geq i_{k+1}$ in case $i_2=i_{k+2}$. 
\end{enumerate}

\item[(4)] $(\F^2)^3=0$. 
\end{enumerate}
\end{lemma}

\begin{lemma}\label{LFid2} \mylabel{LFid2}
\begin{enumerate} 
\item[(1)] The element $(y_1\cdots y_n)(y_1y_2)$ with $n\geq3$ does not belong to the $T$-ideal, generated by $\{(1\cdots k)(12)\,|\,k\geq3, k\neq n\}$ in $\F$. In particular, $(y_1\cdots y_n)(y_1y_2)\neq0$ in $\F$.

\item[(2)] The result of the complete or a partial linearization of $(y_1\cdots y_n)(y_1y_2)$ in $\F$ is zero.

\item[(3)] For every non-zero identity $f\in T[\gl_2]\lhd \F$ we have $f\in (\F^2)^2$.

\item[(4)] Assume that a non-zero identity $f\in T[\gl_2]\lhd \F$ is  not multilinear. Then $f$ an element of the T-ideal of $\F$ generated by $(1\cdots n)(12)$, where $n\geq2$.
\end{enumerate}
\end{lemma}

\section{Generating set}

We use notations and conventions from Section~2.  For short, we denote a product $n(n+1)\ldots (s-1) (s+1) \cdots t$ by $n\cdots \hat{s}\cdots t$, where $r\leq s\leq t$. 

\begin{lemma}\label{L9mine} \mylabel{L9mine}
Every identity $f\in T[\gl_2]\lhd \F$ of multidegree $1^n$ can be represented as the following sum:
\begin{eq}\label{eqf}
f=\sum_{4\leq i\leq n} \al_{i2} (i 3 4 \cdots \hat{i} \cdots n)(21) + 
\sum_{3\leq i\neq j\leq n} \al_{ij} (i 2 3  \cdots \hat{i} \cdots \hat{j}\cdots n)(j1),
\end{eq}
where $\al_{ij}\in\FF$. In particular, $n\geq4$.
\end{lemma}
\begin{proof} We work in $\F$.  We start with the case of  an element  $w=(i_1\cdots i_k) (pq)$ of  multidegree $1^n$ that satisfies conditions from part~3 of Lemma~\ref{LF}, namely, $k\geq2$, $i_1>i_2\leq i_3\cdots \leq i_k$, $p>q\leq i_3$, $i_2\geq q$. Hence $q$ is the least element of $\{1,\ldots,n\}$, i.e., $q=1$. Thus $w=(i_1\cdots i_k) (p1)$, where $\{i_1,\ldots,i_k,p\}=\{2,\ldots, n\}$ and $i_1>i_2\leq i_3\cdots \leq i_k$. Moreover, 
\begin{enumerate}
\item[$\bullet$] if $p=2$, then $w=(i 3 4 \cdots \hat{i} \cdots n)(21)$ for some  $4\leq i\leq n$;

\item[$\bullet$] if $3\leq p\leq n$, then $w=(i 2 3  \cdots \hat{i} \cdots \hat{j}\cdots n)(j1)$ for some $3\leq i\leq n$, $i\neq p$.
\end{enumerate}

Consider an element $f$ from the formulation of lemma. By part~3 of Lemma~\ref{LFid2} we have  that $f\in(\F^2)^2$. Part~3 of Lemma~\ref{LF}  together with the above reasoning completes the proof.
\end{proof}

Given  $f=f(x_1,\ldots,x_n)\in\F$, denote by $f^{(i,j)}$ the result of the following substitution in $f$: $x_i\to b$, $x_j\to c$ and $x_k\to a$ for all $1\leq k\leq n$ with $k\neq i,j$. 
It is easy to verify the following claim.

\begin{lemma}\label{Lfact1} \mylabel{Lfact1} An element $f\in\F$ given by formula~\Ref{eqf} is an identity for $\gl_2$ if and only if $f^{(i,j)}=0$ for all $1\leq i<j\leq n$.
\end{lemma}

\begin{lemma}\label{Lfact2}\mylabel{Lfact2} An element $f\in\F$ given by formula~\Ref{eqf} is an identity for $\gl_2$ if and only if the coefficients $\{\al_{ij}\}$ satisfy the following conditions: 
\begin{eq}\label{eq1-Lmain}
\al_{sr}=\al_{rs} \quad\text{ for all }\quad 3\leq s<r\leq n;
\end{eq}
\begin{eq}\label{eq2-Lmain}
\al_{r2}=\sum_{3\leq j\neq r\leq n}\al_{rj} \quad\text{ for all }\quad 4\leq r\leq n.
\end{eq}

\end{lemma}
\begin{proof}
By the condition of Lemma~\ref{L9mine} we have that $n\geq4$. Note that it could be more convenient for the reader to consider the case of $n=4$ separately (see Example~\ref{ex-n4} below).

Assume that $3\leq s<r\leq n$. Then the equality $f^{(s,r)}=0$  is equivalent to~\Ref{eq1-Lmain}.

Assume that $4\leq r\leq n$. Then the equality $f^{(1,r)}=0$ is equivalent to~\Ref{eq2-Lmain}. 

\noindent{}The equality $f^{(2,r)}=0$ is equivalent to 
$$\al_{r2}+\sum_{3\leq i\neq r\leq n}\al_{ir}=0,$$
which is a linear combination of~\Ref{eq1-Lmain} and~\Ref{eq2-Lmain}.

The equality  $f^{(1,3)}=0$ is equivalent to
\begin{eq}\label{eq3-Lmain}
\sum_{4\leq i\leq n} \al_{i2} + \sum_{3< j\leq n} \al_{3j}=0.
\end{eq}%
Denote the second sum of~\Ref{eq3-Lmain} by $A$.  Applying~\Ref{eq2-Lmain} to $\al_{i2}$ we obtain that~\Ref{eq3-Lmain} is a linear combination of~\Ref{eq2-Lmain} and
$\sum\al_{ij}=A$, where the sum ranges over all $4\leq i\leq n$, $3\leq j\leq n$ with $i\neq j$. The left hand side of this equality is equal to
$$\sum_{4\leq i\leq n}\al_{i3} + \sum_{4\leq i\neq j\leq n} \al_{ij}=/\text{see }\Ref{eq1-Lmain}/= A + 2\!\!\!
\sum_{4\leq i<j\leq n}\al_{ij}.$$
Thus \Ref{eq3-Lmain} is a linear combination of~\Ref{eq1-Lmain} and~\Ref{eq2-Lmain}.

The equality $f^{(2,3)}=0$ is equivalent to  
$$\sum_{4\leq i\leq n} \al_{i2} + \sum_{3< i\leq n} \al_{i3}=0,$$
which is a linear combination of~\Ref{eq1-Lmain} and~\Ref{eq3-Lmain}.

The equality $f^{(1,2)}=0$ is equivalent to 
$$\sum_{3\leq i\neq j\leq n} \al_{ij}=0.$$
Since the left hand side of this equality is 
$$\sum_{3\leq i<j\leq n} (\al_{ij} + \al_{ji})$$
it follows from~\Ref{eq1-Lmain}. The claim is proven.
\end{proof}

For $n\geq 4$ denote
$$f_n = (12)(34\cdots n)+(13)(24\ldots n) +(14)(23\cdots n) \in\F.$$
In particular, $f_4=(12)(34)+(13)(24)+(14)(23)$.

\begin{example}\label{ex-n4} \mylabel{ex-n4} In this example we illustrate the proof of Lemma~\ref{Lfact2} in case $n=4$. Assume that $0\neq f\in\F$ is an identity for $\gl_2$ of multidegree $1^4$. By Lemma~\ref{L9mine}, $f=\al_{42} (43)(21) + \al_{43} (42)(31) + \al_{34}(32)(41)$. Since $f^{(1,2)}=0$, then $\al_{43}+\al_{34}=0$. Since $f^{(1,3)}=0$, then $\al_{42}+\al_{34}=0$. Since $f^{(1,4)}=0$, then $\al_{42}+\al_{43}=0$. Thus,  the coefficients $\{\al_{ij}\}$ satisfy conditions~\Ref{eq1-Lmain} and~\Ref{eq2-Lmain} for $n=4$. Moreover $f=\al_{42} f_4$. By straightforward computations (or applying Lemma~\ref{Lfact1}) we can verify that $f_4$ is actually an identity for $\gl_2$. Therefore modulo multiplication by a constant, $f=f_4$.
\end{example}

\begin{lemma}\label{Lmultlin} \mylabel{Lmultlin}
Every identity $f\in T[\gl_2]\lhd \F$ of multidegree $1^n$ is a linear combination of identities $f_n(x_{\si(1)},\ldots,x_{\si(n)})=0$ for $n\geq 4$ and $\si\in S_n$.
\end{lemma}
\begin{proof}By Lemma~\ref{L9mine}, we have that $f$ is given by formula~\Ref{eqf} for some coefficients $\al_{ij}\in\FF$. It is convenient to form the vector $\un{\al}=(\al_{r2}, \al_{ij}\,|\,4\leq r\leq n,\; 3\leq i\neq j\leq n)$. Since $f$ is uniquely determined by $\un{\al}$, we denote $f$ by $f_{\un{\al}}$. 
By Lemma~\ref{Lfact2}, coefficients $\{\al_{ij}\}$ satisfy  conditions \Ref{eq1-Lmain} and~\Ref{eq2-Lmain}. 
Hence coefficients $\{\al_{sr}\,|\,3\leq s<r\leq n\}$ are ``free'' variables and the rest of coefficients are linear combinations of them. If $\al_{sr}=1$ for some $3\leq s<r\leq n$  and  $\al_{ij}=0$ for all $3\leq i<j\leq n$ with $(i,j)\neq (s,r)$, then we denote $f_{\un{\al}}$ by $f_{(sr)}$. Obviously, for each $\ga,\de\in\FF$, we have  $\ga f_{\un{\al}} + \delta f_{\un{\be}} =f_{ \ga \un{\al} + \de \un{\be}}$. Thus any identity  $f_{\un{\al}}$ is a linear combination of identities $f_{(sr)}$ with $3\leq s<r\leq n$.

To compute $f_{(pq)}$ for $1\leq p<q\leq n$ we point out that $f_{(pq)}=f_{\un{\al}}$, where the only non-zero coefficients from $\{\al_{ij}\}$ are
\begin{enumerate}
\item[$\bullet$] $\al_{pq}=\al_{p2}=\al_{q2}=1$ in case $p\neq 4$;

\item[$\bullet$] $\al_{3q}=\al_{q2}=1$ in case $p=3$;
\end{enumerate}

Assume $p=3$. Then 
$$f_{(pq)}=(q 34 \cdots \hat{q} \cdots n)(21) + 
(324 \cdots \hat{q} \cdots n)(q1) + 
(q 2 4 \cdots \hat{q} \cdots n) (31).$$
For this proof it is convenient to rewrite $f_n$ in $\F$ as
$$(435\cdots n)(21)+(325\ldots n)(41) +(425\cdots n)(31).$$
\noindent{}In case $q=4$ we have $f_{(pq)}=f_n$ in $\F$. In case $q\geq5$ the result of substitution $x_4\to x_q$, $x_q\to x_4$ in $f_n$ is equal to $f_{(pq)}$ in $\F$, because the identity from part~(2) of Lemma~\ref{LFid} holds in $\F$. 

Assume $p\neq 3$. Then
$$f_{(pq)}=(p 34 \cdots \hat{p} \cdots n)(21) + 
(q 34 \cdots \hat{q} \cdots n)(21) + $$
$$(p23 \cdots \hat{p}\cdots \hat{q} \cdots n)(q1) + 
(q 2 3 \cdots \hat{p} \cdots \hat{q} \cdots n) (p1).$$
\noindent{}Denote by $h$ the result of substitution $x_3\to x_q$, $x_q\to x_3$ in $f_n$. Applying identity from part~(2) of Lemma~\ref{LFid} we obtain that 
$$h=(4q 35 \cdots \hat{q}\cdots n) (21) + (q 235 \cdots \hat{q}\cdots n)(q1) + (4235\cdots \hat{q}\cdots n)(q1).$$

In case $p=4$ we have that
$$f_{(pq)}+h=(435\cdots n)(21) + (q345\cdots \hat{q}\cdots n)(21) + (4q25 \cdots \hat{q}\cdots n)(21).$$
Applying part~(2) of Lemma~\ref{LFid} we can rewrite the first summand as 
$$(34q5\cdots \hat{q}\cdots n) (21)= 
(4q35\cdots \hat{q}\cdots n) (21) + 
(q345\cdots \hat{q}\cdots n) (21)$$ 
(see also the Jacobi identity). Thus $f+h=0$ in $\F$ and the claim is proven.

Assume $p\geq 5$. Denote by $g$ the result of substitution $x_4\to x_p$, $x_p\to x_4$ in $f_n$. Applying identity from part~(2) of Lemma~\ref{LFid} we obtain that $g$ is equal to
$$(pq34\cdots \hat{p}\cdots \hat{q} \cdots n)(21) +
(q234\cdots \hat{p}\cdots \hat{q}\cdots n)(p1)+
(p234\cdots \hat{p}\cdots \hat{q}\cdots n)(q1).
$$ 
Then $f_{(pq)}+g$ is equal to
$$(p34\cdots\hat{p}\cdots n)(21) + 
(q34\cdots\hat{q}\cdots n)(21) +
(pq34\cdots\hat{p}\cdots\hat{q}\cdots n)(21).$$
Applying part~(2) of Lemma~\ref{LFid} to the first summand and the second summand we obtain $(3pq4\cdots\hat{p}\cdots\hat{q}\cdots n)(21)$ and $(q3p4\cdots\hat{p}\cdots\hat{q}\cdots n)(21)$, respectively. Then the Jacobi identity implies that $ f_{(pq)}+g=0$. The claim is proven.
\end{proof}

Now we can proof Theorem~\ref{theo-main}.
\begin{proof}
At the end of paper~\cite{VaughanLee_1970} it is shown that the T-ideal $T[\gl_2]\lhd \F$ is generated by 
\begin{enumerate}
\item[$\bullet$] some multilinear identities;

\item[$\bullet$] identities $(12)(12\cdots n)$ for all $n\geq3$.
\end{enumerate}
Lemma~\ref{Lmultlin} and the definition of $\F$ together with part~(1) of Lemma~\ref{L1e2} conclude the proof.
\end{proof}


\end{document}